\documentclass{article}
\usepackage{graphicx} 
\usepackage{amsmath}
\usepackage{amsthm}
\usepackage{amssymb}
\usepackage{amsfonts}
\usepackage{algpseudocode}
\usepackage[nice]{nicefrac}
\usepackage{url} 
\usepackage{authblk}

\newtheorem{theorem}{Theorem}
\newtheorem{claim}[theorem]{Claim}
\newtheorem{corollary}[theorem]{Corollary}
\newtheorem{definition}[theorem]{Definition}

\newtheorem{lemma}[theorem]{Lemma}
\newtheorem{remark}[theorem]{Remark}
\newtheorem{proposition}[theorem]{Proposition}

\title{Entropy and the growth rate of universal covering trees}
\author[1]{Idan Eisner\thanks{eisnerida@telhai.ac.il}}
\author[1]{Shlomo Hoory\thanks{hooryshl@telhai.ac.il}}
\date{October 2026}
\affil[1]{Tel-Hai University, Kiryat Shmona, Israel}

\newcommand{\dirE}{\vec{E}}
\newcommand{\R}{\mathbb{R}}  
\newcommand{\N}{\mathbb{N}} 
\newcommand{\E}{\mathbb{E}} 

\DeclareMathOperator{\var}{var}
\DeclareMathOperator{\outdeg}{outdeg}
\DeclareMathOperator{\indeg}{indeg}
\DeclareMathOperator{\mindeg}{mindeg}
\DeclareMathOperator{\maxdeg}{maxdeg}

\DeclareMathOperator{\dom}{dom}

\begin{document}

\maketitle

\begin{abstract}
    This work studies the relation between two graph parameters, $\rho$ and $\Lambda$.
    For an undirected graph $G$,  $\rho(G)$ is the growth rate of its universal covering tree, 
    while $\Lambda(G)$ is a weighted geometric average of the vertex degree minus one, 
    corresponding to the rate of entropy growth for the non-backtracking random walk (NBRW).    
    It is well known that $\rho(G) \geq \Lambda(G)$ for all graphs, and that 
    graphs with $\rho=\Lambda$ 
    exhibit some special properties. In this work we derive an easy to check,
    necessary and sufficient condition for the equality to hold. 
    Furthermore, we show that the variance of the number of random bits used by a length 
    $\ell$ NBRW is $O(1)$ if $\rho = \Lambda$ and $\Omega(\ell)$ if $\rho > \Lambda$.
    As a consequence we exhibit infinitely many 
    non-trivial examples of graphs with $\rho = \Lambda$.
\end{abstract}

\section{Introduction}

\subsection{The Main Players}
    This paper investigates the relation between two graph parameters $\Lambda$ and $\rho$. 
    Given the graph $G=(V,E)$, which will always be connected, undirected, with minimal degree at least two, and maximal degree at least three, 
    we define $\Lambda(G)$ and $\rho(G)$ as follows:
    \begin{itemize}
    \item 
        $\Lambda$ is the geometric mean of the vertex degrees minus one, where each vertex has weight proportional to its degree:
        \begin{eqnarray*}
            \Lambda(G) = \prod_{v \in V} (\deg(v)-1)^\frac{\deg(v)}{2 |E|}.
        \end{eqnarray*}
    \item 
        $\rho$  is the growth rate of the universal cover of $G$, exact definition can be found in subsection~\ref{subsection:universal}.
    \end{itemize}
    It is well known that $\rho(G) \geq \Lambda(G)$, see~\cite{alon2002moore,hoory2024girth} or subsection~\ref{subsection:Lambda} here.
    In this work we give a simple, necessary and sufficient condition for equality to hold.
    However, before stating the exact result, we begin with some background and motivation. 
    
\subsection{Some Motivation}
    The non-backtracking random walk (NBRW) on the graph $G$ is a process that explores a graph by moving from vertex to vertex, 
    where in 
    each step, the walk moves to a randomly chosen neighbor of the current vertex, excluding the one it just visited. 
    It was shown by Alon et al.~\cite{alon2007non} that this extra bias towards exploration makes the walk mix faster on regular expander graphs, compared to the simple random walk (SRW) that does not employ this extra constraint. 
    Also, it was shown that the SRW as well as the NBRW exhibit a sharp threshold in their convergence to the stationary distribution for regular Ramanujan graphs~\cite{lubetzky2016cutoff}, as well as for various random graph models with high probability.
    Lubetzky et al.~\cite{lubetzky2010cutoff}, proved such a result for random $d$-regular graphs, ${\cal G}(n,d)$,
    Ben-Hamou et al.~\cite{ben2017cutoff}, for random graphs with a prescribed degree distribution in the configuration model,
    and Conchon-Kerjan~\cite{conchon2022cutoff}  for random lifts of a fixed base graph.

    The case of a random $n$-lift of some base graph $G$ is closely related to this work, as the mixing time of the NBRW is $t_s = (1+o(1)) \log_{\Lambda}n$ with high probability~\cite{conchon2022cutoff}, while the diameter of a random $n$-lift is $t_d = (1+o(1)) \log_{\rho}n$ with high probability~\cite{conchon2021sparse}. 
    Put differently, performing the NBRW on a random $n$-lift of $G$, the first time when all vertices have positive probability of being reached is $t_d$, 
    while the first time when the walk gets close to the stationary distribution is $t_s$.
    Clearly $t_s \geq t_d$.
    Therefore, asking if $\rho(G) = \Lambda(G)$ is the same as asking if the two events occur at the same time, up to a $1+o(1)$ factor. 

    A second motivation for the $\rho$ vs. $\Lambda$ question is the Moore bound for irregular graphs.
    In \cite{alon2002moore}, Alon et al. proved that any girth $g$ graph 
    with the same degree distribution as $G$ has at least $\Omega(\Lambda(G)^{g/2})$ vertices,
    while in \cite{hoory2024girth} Hoory proved that any girth $g$ graph covering the base graph $G$ has at least $\Omega(\rho(G)^{g/2})$ vertices.
    Therefore, when $\rho(G) = \Lambda(G)$ the bounds coincide, and otherwise there is an exponential gap between the two bounds.

\subsection{Our Results}\label{subsec:results}

    For the graph $G$, define a \emph{suspended path} as a non-backtracking path $P$ in $G$ 
    that has no internal vertices of degree more than two, and where its two endpoints $u,v$ have degree greater than two.   
    The main result of this paper is that a necessary and sufficient condition for $G$ to have $\Lambda = \rho$, 
    is that the following equation holds for all suspended paths $P$ in $G$, where $|P|$ denotes the number of edges in $P$:
    \begin{eqnarray}\label{eq:suspended_path_condition_intro}
        (\deg(u)-1) \cdot (\deg(v)-1) = \Lambda(G)^{2|P|}.
    \end{eqnarray}    

    Graphs satisfying $\rho=\Lambda$ include regular graphs and bipartite-biregular graphs.
    In addition, any such graph, where all edges are replaced by length $k$ paths for some fixed $k>1$ also satisfies the condition.
    It turns out that the simple combinatorial characterization given in~\eqref{eq:suspended_path_condition_intro}, 
    can be used to construct infinitely many graphs with $\rho=\Lambda$ that are not in the above list, such as those in Figure~\ref{fig:lambda_rho_counterexample}.
    \begin{figure}[h]
        \centering
        \includegraphics[width=1.0\textwidth]{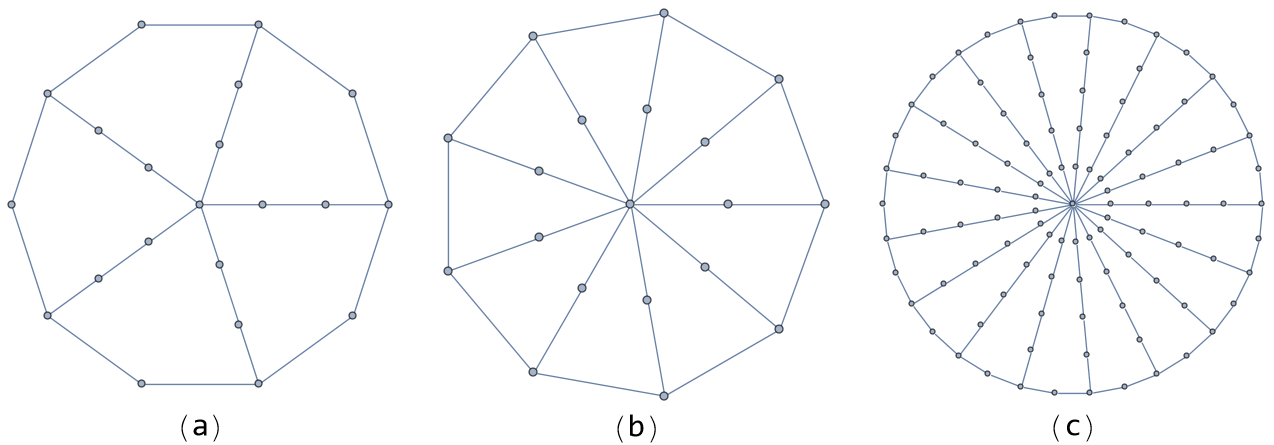}
        \caption{
            Three graphs with $\rho = \Lambda$ 
            that are not regular, bipartite bi-regular,
            or obtained from such graphs by 
            replacing each edge by a length $k$ path.
           It can be verified that the suspended path condition holds with $\Lambda$ equals $\sqrt{2}$, $2$ and $\sqrt{2}$ for the graphs (a), (b) and (c), respectively.
        }
        \label{fig:lambda_rho_counterexample}
    \end{figure}

    The third result concerns the random variable $R_\ell$ counting the number of random bits consumed by a length $\ell$ NBRW starting from the stationary distribution.
    Namely, the sum of $\log_2(\deg(v)-1)$ on the vertices visited by the walk.
    It is not hard to verify that the expected value of $R_\ell$ is $\ell \log_2 \Lambda$. 
    We prove that the following dichotomy on the variance of $R_\ell$ holds:
    If $\rho(G)=\Lambda(G)$ then $\var[R_\ell]$ is bounded by a constant independent of $\ell$,
    but if $\rho(G)>\Lambda(G)$ then the variance grows linearly with $\ell$.

    The rest of the paper is organized as follows:
    In Section~\ref{sect:preliminaries} we give the required background for the main body of this work.
    In Section~\ref{sect:main_theorem} we prove the main theorem, characterizing the graphs for which $\Lambda = \rho$.
    In Section~\ref{sect:graphs_with_rho_equal_lambda} we exhibit a family of graphs with $\Lambda = \rho$ that do not fall into the set of known examples. 
    In Section~\ref{sect:variance} we prove that the variance of the number of random bits used follows a dichotomy,     and we demonstrate the result on the $K_4$ minus an edge graph.
    We end with a set of open questions.
\section{Preliminaries}\label{sect:preliminaries}
\subsection{The Non-Backtracking Random Walk and Adjacency Matrix}\label{subsection:nbrw}
Given a simple undirected graph $G=(V,E)$, we regard each edge as a pair of directed edges and denote the set of directed edges by $\dirE$.
For a directed edge $e=(u,v) \in \dirE$, we denote its head by $h(e) = v$ and tail by $t(e) = u$.
We define $\outdeg(e) = \deg(h(e))-1$, $\indeg(e) = \deg(t(e))-1$, and the reversed edge $e' = (u,v)' = (v,u)$. 
\begin{definition}
    Let $G=(V,E)$ be a simple undirected graph with minimal degree at least two.  
    For two directed edges $e,f$, we say that one can transition from $e$ to $f$ if $h(e)=t(f)$ and $e' \neq f$ and denote this relation by $e \rightarrow f$.
    The Non-Backtracking Random-Walk (NBRW) on $G$ is a Markov chain on the state space $\dirE$ with the transition matrix $\Pi_G$ defined by:
    \begin{equation*}
        \pi_G(e,f) = \begin{cases}
             \frac{1}{\outdeg(e)}, & \mbox{ if }e \rightarrow f\\
             0,                    & \mbox{otherwise}
        \end{cases}
    \end{equation*}
    We regard the state $e=(u,v)$ as being at $v$ coming from $u$. 
    Let $B$ denote the non-backtracking adjacency matrix, which is an $\dirE$ by $\dirE$ 0-1 matrix where $B_{e,f}=1$ if $e \rightarrow f$.
\end{definition}

\begin{remark}
    The definition generalizes to allow for multiple edges and self loops in $G$, where a self loop can be either a whole-loop or a half-loop. 
    Each edge, except for the half-loops, gives rise to two edges in $\dirE$ each being the inverse of the other.  
    A half-loop in $G$ gives rise to a single edge in $\dirE$, which is its own inverse.
\end{remark}

\begin{claim}
    For any graph $G=(V,E)$ with minimal degree at least two, 
    the uniform edge distribution $\nu_s$ defined by $\nu_s(e) = 1/|\dirE|$ for all $e \in \dirE$ is stationary for the Markov chain $\Pi_G$.
\end{claim}
\begin{proof}
    \begin{eqnarray*}
        (\nu_s \Pi_G)_f 
            &=& \sum_{e \in \dirE} \nu_s(e) \pi_G(e,f) = \frac{1}{|\dirE|} \sum_{e \in \dirE} \pi_G(e,f)\\
            &=& \frac{1}{|\dirE|} \sum_{e: e \rightarrow f} \frac{1}{\outdeg(e)} = \frac{1}{|\dirE|} = \nu_s(f).
    \end{eqnarray*}
\end{proof}

\noindent
The next claim is known, but we include a proof for completeness' sake.

\begin{claim}[\cite{glover2021non} proposition 3.3]\label{claim:glover2021non}
    The NBRW Markov chain on $G$ as well as the non-backtracking adjacency matrix $B$ are irreducible iff $G$ is connected with $\mindeg(G) \geq 2$ and $\maxdeg(G) > 2$.
\end{claim}
\begin{proof}
    To prove that the condition on $G$ is required, assume that the NBRW on $G$ is irreducible. 
    The condition $\mindeg(G) \geq 2$ must hold since the walk is well defined. 
    In addition, $G$ must be connected, because otherwise one can find edges $e, f \in \dirE$ with no non-backtracking path from $e$ to $f$. 
    It remains to exclude the possibility that $G$ is connected with $\mindeg(G) = \maxdeg(G) = 2$, namely that $G$ is a cycle. 
    Indeed, in this case, for any $e \in \dirE$, there is no non-backtracking path from $e$ to $e'$.
    
    To prove that the condition on $G$ implies irreducibility, 
    we first argue that it suffices to find a path $P_{h,h'}$ from some edge $h$ to its inverse $h'$ to imply irreducibility.
    Indeed, given any pair of edges $e,f \in \dirE$, because $G$ is connected 
    there is a non-backtracking path $P_1$ connecting $e$ to either $h$ or $h'$ and 
    there is a non-backtracking path $P_2$ connecting $h$ or $h'$ to $f$. 
    Therefore, there is a non-backtracking path from $e$ to $f$ formed by concatenating $P_1$, possibly $P_{h,h'}$ or its inverse, and $P_2$.

    To conclude, observe that the condition $\mindeg(G) \geq 2$ implies the existence of some non-backtracking cycle $C = e_1, e_2, \ldots, e_\ell$. 
    As the graph is connected with $\maxdeg(G) > 2$, there must be a vertex $v \in C$ with $\deg(v)>2$,
    and we may assume that $t(e_1)=h(e_\ell)=v$, as depicted in Figure~\ref{fig:vertex_deg3}.
    \begin{figure}[h]
        \centering
        \includegraphics[width=0.3\textwidth,clip,trim=1mm 1mm 1mm 1mm]{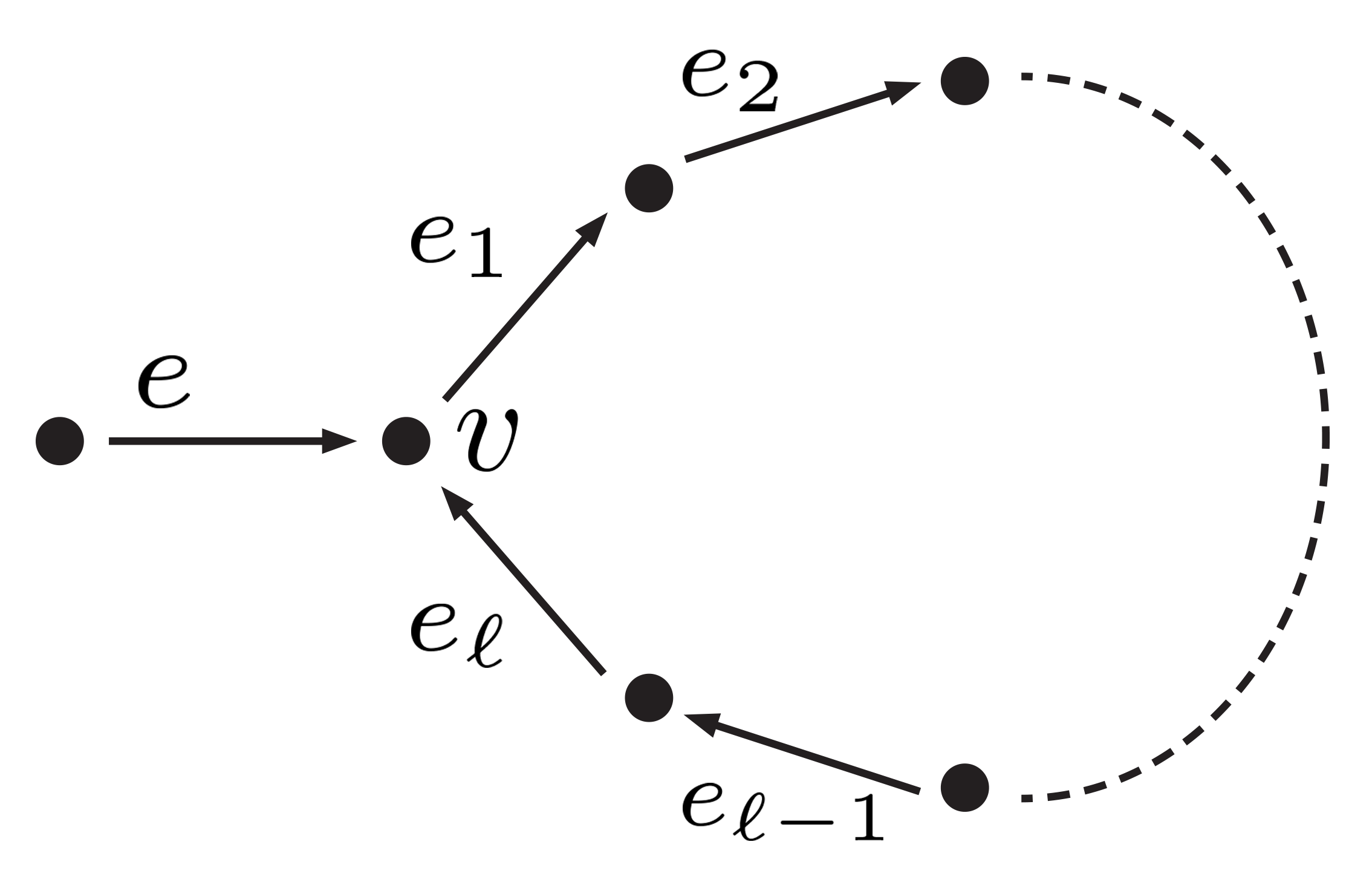}
        \caption{A cycle with a vertex of degree more than two.}
        \label{fig:vertex_deg3}
    \end{figure}
    Since $\deg(v)>2$ there is some edge $e$ with $h(e)=v$ distinct from $e_1'$ and $e_\ell$. Then the path $e,e_1, e_2, \ldots, e_\ell,e'$ is a path from $e$ to $e'$, implying the Markov chain irreducibility.
\end{proof}

Inspired by Claim~\ref{claim:glover2021non} we will say that a graph $G$
is \emph{NB-irreducible} if $G$ is connected with $\mindeg(G) \geq 2$ and $\maxdeg(G) > 2$.

\subsection{The Universal Covering Tree and its Growth Rate $\rho$}\label{subsection:universal}
Given a connected graph $G$ with minimal degree at least 2 and some arbitrary vertex $v_0 \in V(G)$, 
we define its universal cover $\tilde{G}$.
It is an infinite tree, where the vertex set $V(\tilde{G})$ is the set of all finite non-backtracking walks from $v_0$, 
and where two vertices of $\tilde{G}$ are adjacent if one walk extends the other by a single step.
Its root is the empty walk $\epsilon$, and it can be verified that 
changing $v_0$ yields an isomorphic tree. We say that a walk has length $\ell$ if it consists of $\ell+1$ edges, i.e., it has $\ell$ non backtracking transitions.

\begin{definition}
    Given a finite connected graph $G$ with minimal degree at least 2, let:
    \begin{itemize}
        \item $\mathcal{B}_{\tilde{G},r}(v)$ denote the radius $r$ ball around the vertex $v$ in $\tilde{G}$, 
        \item $\Omega_{e,\ell}$ be the set of length $\ell$ non-backtracking walks in $G$ starting from $e \in \dirE$,
        \item $\Omega_\ell = \cup_{e \in \dirE} \Omega_{e,\ell}$ be the set of all length $\ell$ non-backtracking walks in $G$.
    \end{itemize}
\end{definition}

\noindent
The growth rate of $\tilde{G}$, denoted $\rho(G)$, is defined by one of the following alternative definitions, given by the following lemma. Recall that for some matrix $A$, its spectral radius $\rho(A)$ is defined as its maximal eigenvalue in absolute value. The Perron-Frobenius Theorem (Theorem 8.4.4 in~\cite{horn2012matrix}) states that for any non-negative irreducible matrix $A$, $\rho(A)>0$ is an algebraically simple eigenvalue with strictly positive left and right eigenvectors. 

\begin{lemma}[\cite{angel2015non,hoory2024girth}]\label{lemma:rho_Perron}
    Given an NB-irreducible graph $G$, the following definitions for $\rho(G)$ are equivalent:        
    \begin{eqnarray*}
        \rho(G) 
            &=& \lim_{r \rightarrow \infty} |\mathcal{B}_{\tilde{G},r}(v)|^{1/r} 
            =   \lim_{\ell \rightarrow \infty} |\Omega_{e,\ell}|^{1/\ell}
            =   \lim_{\ell  \rightarrow \infty} |\Omega_\ell|^{1/\ell} \\
            &=& \lim_{\ell \rightarrow \infty} \left( \nu_s \ B^\ell \ \overline{1} \right)^{1/\ell}
            = \rho(B),
    \end{eqnarray*}
    where $B$ denotes the non-backtracking adjacency matrix, $\rho(B)$ is the largest (Perron) eigenvalue of $B$, and $\overline{1}$ is the all-ones vector.
    Equality holds regardless of the choice of $v \in V(\tilde{G})$ and $e \in \dirE$, where applicable.
\end{lemma}

We elaborate on why the last equality holds, as this argument is repeatedly used for different matrices throughout the paper. This equality is a consequence of Gelfand's formula, the fact that the entry-wise $\ell_1$ norm is a matrix norm (5.6.0.1) in \cite{horn2012matrix}, and that $\overline{1}^t A \overline{1}$ is the same as the entry-wise $\ell_1$ norm for any non-negative matrix $A$.
\begin{proposition}[Gelfand formula, Corollary 5.6.14 in \cite{horn2012matrix}]\label{proposition:Gelfand}
    Let $\| \cdot \|$ be a matrix norm on $M_n$ and let $A\in M_n$.
    Then $\rho(A)=\lim_{k\to\infty}\|A^k\|^{1/k}$.
\end{proposition}

\subsection{The Average Growth Rate $\Lambda$ and its Connection to $\rho$}\label{subsection:Lambda}
\begin{definition}
   The average growth rate $\Lambda(G)$ is the rate predicted by the stationary distribution $\nu_s$ of the NBRW:
\begin{equation}\label{eq:Lambda}
    \Lambda(G) = \Bigl( \prod_{e \in \dirE} \outdeg(e) \Bigr)^{1/{|\dirE|}} = \prod_{v \in V(G)} \Bigl( \deg(v)-1 \Bigr)^{\deg(v)/{|\dirE|}}.
\end{equation} 
\end{definition}

When no confusion may arise, we drop the graph argument and use $\rho$ and $\Lambda$ instead of $\rho(G)$ and $\Lambda(G)$. 
One should note that while the graph parameter $\Lambda$ depends only on the degree distribution of the graph,
the parameter $\rho$ depends on the graph structure as well. 
The average growth rate $\Lambda$ played a major role in~\cite{alon2002moore,hoory2002size} 
and corresponds to the average rate of entropy growth for the NBRW.

In order to continue, note that the NBRW on $G$, with initial state drawn from the stationary distribution $\nu_s$, 
induces a probability measure $\mu$ on the set $\Omega_\ell$ of length $\ell$ non-backtracking walks on $G$. 
The probability of the walk $\omega = (e_0,e_1,\ldots,e_\ell)\in \Omega_\ell$ is: 
\begin{equation}
    \mu[\omega] = \frac{1}{|\dirE|} \prod_{i=0}^{\ell-1} \frac{1}{\outdeg(e_i)}
\end{equation}

Given a strictly positive function $\beta:\dirE\rightarrow\R^+$ 
and a walk $\omega = (e_0,e_1,\ldots,e_\ell)\in \Omega_\ell$, 
define the random variable 
$X_\beta(\omega) = \prod_{i=0}^{\ell-1} \beta(e_i)$.

\begin{lemma}\label{lemma:Xbeta_lb}
    \begin{equation*}
        \E_{\Omega_\ell}[X_\beta] \geq \left( \prod_{e \in \dirE} \beta(e) \right)^{\frac{\ell}{|\dirE|}}
    \end{equation*}
\end{lemma}

\begin{proof}
    \begin{eqnarray*}
        \E_{\Omega_\ell}[X_\beta] 
        &=& \sum_{\omega \in \Omega_\ell} \mu[\omega] X_\beta(\omega) 
        \geq \exp\left( \E_{\Omega_\ell}[\log X_\beta]\right)\\
        &=& \exp\left( \frac{\ell}{|\dirE|} \sum_{e \in \dirE} \log \beta(e)\right)
        = \left( \prod_{e \in \dirE} \beta(e) \right)^\frac{\ell}{|\dirE|},        
    \end{eqnarray*}
    where the inequality is Jensen's inequality, and the equality before the last follows by linearity of expectation and the fact that the marginal distribution of $e_i$ is $\nu_s$ for all $i$.
\end{proof}

\begin{lemma}\label{lemma:ModifiedPiPerronLowerBound}
    For an NB-irreducible graph $G$, define the modified NBRW Markov chain $\Pi_\beta$ by:
    \begin{equation*}
        (\Pi_\beta)_{e,f} = (\Pi_G)_{e,f} \cdot \beta(e).
    \end{equation*}
    Then the Perron eigenvalue of $\Pi_\beta$ satisfies:
    \begin{equation}\label{eq:ModifiedPiPerronLowerBound}
        \rho(\Pi_\beta) 
        =    \lim_{\ell \rightarrow \infty} \left( \E_{\Omega_\ell}[X_\beta] \right)^\frac{1}{\ell} 
        \geq \left( \prod_{e \in \dirE} \beta(e) \right)^\frac{1}{|\dirE|}.
    \end{equation}
\end{lemma}

\begin{proof}
    \begin{eqnarray*}
         \nu_s \left( \Pi_\beta \right)^{\ell} \overline{1} 
         &=& \sum_{\omega = (e_0,e_1,\ldots,e_\ell) \in \Omega_\ell} \mu[\omega] \ \prod_{i=0}^{\ell-1} \beta(e_i) 
         = \E_{\Omega_\ell}[X_\beta].
    \end{eqnarray*}
    Therefore by Proposition~\ref{proposition:Gelfand},
    \begin{equation*}
         \rho(\Pi_\beta) 
         = \lim_{\ell \rightarrow \infty} \left( \nu_s \ ( \Pi_\beta )^\ell \ \overline{1} \right)^{1/\ell}
         =   \lim_{\ell \rightarrow \infty} \E_{\Omega_\ell}[X_\beta]^{1/\ell}.
    \end{equation*}
    The result follows by applying Lemma~\ref{lemma:Xbeta_lb} to lower bound the right hand side.
\end{proof}

\noindent
When plugging $\beta(e) = \outdeg(e)$, we get $\Pi_\beta = B$.
Therefore, Lemma~\ref{lemma:ModifiedPiPerronLowerBound} with $\beta(e) = \outdeg(e)$ yields the inequality:
\begin{corollary}\label{cor:rho_ge_lambda}
     $\rho \geq \Lambda$.    
\end{corollary}

\noindent
Recall the random variable $R_\ell$ counting the number of bits used by the length $\ell$ NBRW starting from the stationary distribution $\nu_s$,
discussed in Section~\ref{subsec:results}. Then:
\begin{corollary}\label{cor:Rl}
    $\Lambda = 2^{\E[R_{\ell}/\ell]}$ and $\log_2 \rho(B) = \lim_{\ell \rightarrow \infty} \frac{1}{\ell} \log_2 \E[2^{R_\ell}]$.
\end{corollary}

\begin{proof}
    By definition, $R_\ell = \log_2(X_\beta)$ with $\beta(e) = \outdeg(e)$. 
    Therefore, by the last two equalities in the proof of Lemma~\ref{lemma:Xbeta_lb}, we have $\E[R_\ell] = \E[\log_2 X_\beta] = \ell \log_2 \Lambda$, which is the first equality.
    
    For the second equality, we use Lemma~\ref{lemma:ModifiedPiPerronLowerBound} with $\beta(e) = \outdeg(e)$, 
    yielding that $\log_2 \rho(B) = \log_2 \lim_{\ell \rightarrow \infty} \left( \E[X_\beta] \right)^\frac{1}{\ell} = \lim_{\ell \rightarrow \infty} \frac{1}{\ell} \log_2 \E[2^{R_\ell}]$.
\end{proof}

\section{The Main Theorem}\label{sect:main_theorem}
    Given a non-backtracking path $P=(e_0,e_1 ,\ldots,e_{\ell-1})$ in an NB-irreducible graph $G$, for some $\ell \geq 1$, we say that $P$ is a \emph{suspended path} if:
    \begin{enumerate}
        \item $\outdeg(e_i) = 1$ for $i<\ell-1$ and $\outdeg(e_{\ell-1}) > 1$
        \item $\indeg(e_i)=1$ for all $i>0$ and $\indeg(e_0) > 1$.
    \end{enumerate}
    The length of the path is $|P|=\ell$ (representing the number of edges in $P$). 
    We denote $\outdeg(P) = \outdeg(e_{\ell-1})$ and $\indeg(P) = \indeg(e_0)$.
    Let $\mathcal{S}=\mathcal{S}(G)$ be the set of all suspended paths of $G$. 
    Note that the suspended paths form a partition of $\dirE$ and that $\mathcal{S}$ is closed under reversal $P^{-1}=(e'_{\ell-1},\ldots,e'_0)$.

\begin{definition}
    The {\bf suspended path condition} holds for the graph $G$ if
    \begin{equation}\label{eq:suscond}
        \outdeg(P) \cdot \indeg(P) = {\Lambda(G)}^{2|P|} \mbox{ for all suspended paths } P \in \mathcal{S}.
    \end{equation}
\end{definition}

A \emph{non-backtracking cycle} of length $k\ge 1$ is a cyclic sequence
$C=(e_0,\dots,e_{k-1})$ of directed edges, not necessarily distinct, such that
$e_j\to e_{j+1}$ for all $j$, where indices are taken modulo $k$. 

\begin{definition}
    The {\bf cycle condition} holds for the graph $G$ if 
    \begin{equation}
        \prod_{j=0}^{k-1} \outdeg(e_j) = \Lambda(G)^{k} \, \mbox{ for every non-backtracking cycle } C.\label{eq:cyclecond}
    \end{equation}    
\end{definition}

\noindent
Following is our main result:

\begin{theorem}\label{theorem:main}
    For an NB-irreducible graph $G$ 
    the following are equivalent:
    \begin{enumerate}
        \item $\rho(G) = \Lambda(G)$. \label{theorem:main_rho=Lambda}
        \item $G$ satisfies the suspended path condition.\label{theorem:main_path}
        \item $G$ satisfies the cycle condition.\label{theorem:main_cycle}
    \end{enumerate}
\end{theorem}

The proof is split into two parts. In Subsection~\ref{subsect:sufficient} we prove that
each of the two combinatorial conditions implies $\rho=\Lambda$. Both proofs go
through the same intermediate property: the product of the out-degrees along
every non-backtracking walk of length $\ell$ is $\Lambda^\ell$ up to a bounded
factor. In Subsection~\ref{subsect:necessary} we prove the converse: if $\rho=\Lambda$,
the log-convexity of the Perron eigenvalue forces the existence of a potential
function on $\dirE$, from which both conditions follow. Finally, in
Subsection~\ref{subsect:relaxed} we give a reformulation of the suspended path
condition that does not require knowing $\Lambda(G)$ in advance.

\subsection{Sufficiency: each condition implies $\rho=\Lambda$}\label{subsect:sufficient}

For a walk $\omega=(e_0,\dots,e_\ell)\in\Omega_\ell$
let
\[
F(\omega)=\prod_{i=0}^{\ell-1}\outdeg(e_i)=2^{R_\ell(\omega)},
\]
where $R_\ell(\omega)=\sum_{i=0}^{\ell-1}\log_2\outdeg(e_i)$ is the number of
random bits used by the NBRW along $\omega$. 

\begin{lemma}\label{lemma:bounded}
    Let $G$ be an NB-irreducible graph, and assume that there are constants
    $0<c_1\le c_2$, depending only on $G$, such that
    \begin{equation}\label{eq:bounded}
    c_1\,\Lambda(G)^\ell\le F(\omega)\le c_2\,\Lambda(G)^\ell
    \qquad\text{for all $\ell\ge 1$ and $\omega\in\Omega_\ell$.}
    \end{equation}
    Then $\rho(G)=\Lambda(G)$. Moreover,
    $|R_\ell(\omega)-\ell\log_2\Lambda(G)|\le\max(\log_2 c_2,-\log_2 c_1)$ for all
    $\ell$ and $\omega\in\Omega_\ell$.
\end{lemma}

\begin{proof}
    By (3), $\mu[\omega]F(\omega)=1/|\dirE|$ for every $\omega\in\Omega_\ell$, and
    therefore
    \[
    \nu_sB^\ell\mathbf{1}=\frac{|\Omega_\ell|}{|\dirE|}
    =\sum_{\omega\in\Omega_\ell}\mu[\omega]\,F(\omega).
    \]
    By \eqref{eq:bounded} the right hand side lies between $c_1\Lambda^\ell$ and
    $c_2\Lambda^\ell$. Taking the $\ell$-th root and letting $\ell\to\infty$,
    Lemma~\ref{lemma:rho_Perron} yields $\rho(G)=\Lambda(G)$. The second statement is
    \eqref{eq:bounded} after taking $\log_2$.
\end{proof}

\begin{lemma}\label{lemma:suspended_implies_rho_eq_Lambda}
    If the NB-irreducible graph $G$ satisfies the suspended path condition, then
    \eqref{eq:bounded} holds, and in particular $\rho(G) = \Lambda(G)$.
\end{lemma}

\begin{proof}
    Given a graph $G$ satisfying the suspended path condition and some non-backtracking walk $\omega = (e_0,\ldots,e_\ell)\in \Omega_\ell$, 
    let $i_1 < \cdots < i_s$ be the set of indices $i<\ell$ with $\outdeg(e_i) > 1$.
    If $s>0$, then $\omega$ can be expressed as the concatenation $P_0 P_1 \cdots P_s$, where $P_j = (e_{i_j+1},e_{i_j+2},\ldots,e_{i_{j+1}})$ for $j=1,\ldots,s-1$ are suspended paths, 
    $P_0 = (e_0,\ldots,e_{i_1})$ and $P_s = (e_{i_s+1},\ldots,e_\ell)$ are partial suspended paths.
    Then:
    \begin{eqnarray}
        F(\omega) 
        &=& \prod_{i=0}^{\ell-1} \outdeg(e_i) 
         =  \prod_{j=0}^{s-1} \outdeg(P_j)                                \nonumber \\
        &=& \outdeg(P_0)^\frac{1}{2} \left[ \prod_{j=1}^{s-1} \indeg(P_j)^\frac{1}{2} \outdeg(P_j)^\frac{1}{2} \right] \indeg(P_s)^\frac{1}{2}                               \nonumber \\
        &=& \outdeg(P_0)^\frac{1}{2} \indeg(P_s)^\frac{1}{2} \cdot \Lambda(G)^{\ell+1-|P_0|-|P_s|}
        = \Theta(1) \cdot \Lambda(G)^{\ell},                              \label{eq:outdegree_prod_approx}
    \end{eqnarray}
    where the third equality uses the fact that $\outdeg(P_j) = \indeg(P_{j+1})$, 
    the fourth equality uses the suspended path condition, 
    and where $\Theta(1)$ denotes a factor that is upper and lower bounded by positive constants that depend on $G$ but not on $\ell$ or $\omega$.
    
    Otherwise, $s=0$ and $F(\omega)=1$. As $\ell$ is bounded by the maximal length of a suspended path, it follows that $F(\omega) = \Theta(1) \cdot \Lambda(G)^{\ell}$.
    
    In both cases we get $\rho(G) = \Lambda(G)$ by Lemma~\ref{lemma:bounded}.
\end{proof}

\begin{lemma}\label{lem:cycle-suff}
    If the NB-irreducible graph $G$ satisfies the cycle condition, then
    \eqref{eq:bounded} holds, and in particular $\rho(G)=\Lambda(G)$.
\end{lemma}

\begin{proof}
    By Claim~\ref{claim:glover2021non} the matrix $B$ is irreducible, so for every $e,f\in\dirE$ there is a
    non-backtracking walk $e=x_0\to x_1\to\dots\to x_m=f$ with $m\ge 1$. Let
    $l_0=l_0(G)$ be the maximum, over all pairs $e,f\in\dirE$, of the minimal such
    $m$.
    
    Let $\omega=(e_0,\dots,e_\ell)\in\Omega_\ell$, and choose a non-backtracking
    walk $e_\ell=x_0\to x_1\to\dots\to x_m=e_0$ with $1\le m\le l_0$. Then
    \[
    C=(e_0,e_1,\dots,e_{\ell-1},x_0,x_1,\dots,x_{m-1})
    \]
    is a non-backtracking cycle of length $\ell+m$. By the cycle condition,
    \[
    \Lambda^{\ell+m}=F(\omega)\cdot\prod_{j=0}^{m-1}\outdeg(x_j).
    \]
    Since both the product $\prod_{j=0}^{m-1}\outdeg(x_j)$ and $\Lambda^m$ are $\Theta(1)$, 
    condition \eqref{eq:bounded} holds, implying that $\rho = \Lambda$ by Lemma~\ref{lemma:bounded}.
\end{proof}

\subsection{Necessity: $\rho=\Lambda$ implies both conditions}\label{subsect:necessary}

In order to prove the converse, we study the way the Perron eigenvalue changes
from $\rho(\Pi_G)=1$ to $\rho(G)=\rho(B)$ as the matrix changes gradually from
$\Pi_G$ to $B$, where $\Pi_G$ denotes the transition matrix of the NBRW Markov
chain and $B$ is the non-backtracking adjacency matrix. The following theorem
gives a necessary and sufficient condition for strict log-convexity of
$\rho(M_t)$ for $0\le t\le 1$, where $M_t$ is an element-wise interpolation
between $\Pi_G$ and $B$. The log-convexity property \eqref{eq:Nussbaum_logconvexity}
already appears in Kingman~\cite{kingman1961convexity}.


\begin{theorem}[Nussbaum~\cite{nussbaum1986convexity}, Corollary 1.2]\label{theorem:Nussbaum}
    For $0 \leq t \leq 1$, assume that $M_t$ is an irreducible non-negative matrix such that $(M_t)_{i,j} = (U_{i,j})^{1-t} (V_{i,j})^t$ for all $i,j$.
    Then:
    \begin{equation}\label{eq:Nussbaum_logconvexity}
        \rho(M_t) \leq \rho(U)^{1-t} \rho(V)^t,
    \end{equation}
    where equality holds for some $0 < t <1$ 
    if and only if there exists a positive diagonal matrix $D$ and a constant $k > 0$ such that
    \begin{equation}\label{eq:Nussbaum_condition}
        U = k D^{-1} V D.
    \end{equation}
    Furthermore, if \eqref{eq:Nussbaum_condition} holds then \eqref{eq:Nussbaum_logconvexity} holds with equality for all $t$.
\end{theorem}

Let $M_t$ be the matrix defined by Theorem~\ref{theorem:Nussbaum} with $U = \Pi_G$ and $V = B$.
As $M_t$ has the same support as $B$, it is irreducible by Claim~\ref{claim:glover2021non}, if $G$ is NB-irreducible.

\begin{lemma}\label{lemma:Mt_eq_Lambda_t}
    Let $G$ be an NB-irreducible graph such that $\rho(G) = \Lambda(G)$. 
    Then  $\rho(M_t) = \Lambda(G)^t$ for all $0 \leq t \leq 1$.
\end{lemma}

\begin{proof}
    We first note that the matrix $\Pi_\beta$ in Lemma~\ref{lemma:ModifiedPiPerronLowerBound}, where $\beta(e)=\outdeg(e)^t$, is equal to $M_t$.
    Therefore \eqref{eq:ModifiedPiPerronLowerBound} yields:
    \begin{equation*}
        \rho(M_t) = \rho(\Pi_\beta) \geq \Lambda(G)^t \, \mbox{ for all }0 \leq t \leq 1.
    \end{equation*}
    On the other hand, by \eqref{eq:Nussbaum_logconvexity} 
    \begin{equation*}
        \rho(M_t) \leq \rho(\Pi_G)^{1-t}\rho(B)^t = \Lambda(G)^t
    \end{equation*}
    where the equality holds since $\rho(\Pi_G) =1$ and 
    $\rho(B) = \Lambda(G)$ by our assumption.
\end{proof}

\noindent
The key step is the existence of a potential function on the directed edges:

\begin{lemma}\label{lem:potential}
    Let $G$ be an NB-irreducible graph such that $\rho(G)=\Lambda(G)$. Then there
    exists a function $\varphi:\dirE\to\mathbb{R}$ such that
    \begin{equation}\label{eq:potential}
        \varphi(f)=\varphi(e)+\log\Lambda(G)-\log\outdeg(e)\qquad\text{for all } e,f \in \dirE \text{ with } e\to f.
    \end{equation}
\end{lemma}

\begin{proof}
    Given an NB-irreducible graph $G$ such that  $\rho(G) = \Lambda(G)$, let the matrix $M_t$ be defined as above.
    Then $\rho(M_t)=\rho(G)^t$ for all $0 \leq t \leq 1$, by Lemma~\ref{lemma:Mt_eq_Lambda_t},
    and therefore by Theorem~\ref{theorem:Nussbaum} we have
    \begin{equation}\label{eq:P_diagsim_B}
    \Pi_G=kD^{-1}BD
    \end{equation}
    where $k>0$ is constant, and $D$ is a diagonal matrix with positive diagonal entries.
    Let $\lambda = \log k$ and define $\varphi:\dirE \to \R$ by 
    $\varphi(e) = \log D_{e,e} $.

    For any pair of edges $e,f$ such that $e \rightarrow f$, we have $B_{e,f} = 1$  and
    $(\Pi_G) _{e,f} = \outdeg(e)^{-1} = \indeg(f)^{-1}$. 
    Taking $\log$ of the $e,f$ entry of equation \eqref{eq:P_diagsim_B} yields:
    \begin{equation}\label{eq:Nsbm_ef}
        -\log\outdeg(e) = \lambda -\varphi(e)+\varphi(f).
    \end{equation}
    Summing equation \eqref{eq:Nsbm_ef} over all $e \rightarrow f$ pairs with weight $\outdeg(e)^{-1}$
    yields:
    \begin{equation}\label{eq:sum_NsbmoverEdges}
        -|\dirE| \log \Lambda = |\dirE | \cdot\lambda -\sum_{e}\varphi(e)+\sum _{f}\varphi(f) 
        = |\dirE | \cdot\lambda.
    \end{equation}
    Hence, $\lambda = -\log \Lambda$. Plugging $\lambda$ into \eqref{eq:Nsbm_ef} yields equation \eqref{eq:potential}, as claimed.
\end{proof}

\begin{lemma}\label{lem:cycle-nec}
    If an NB-irreducible graph $G$ satisfies $\rho(G)=\Lambda(G)$, then it
    satisfies the cycle condition \eqref{eq:cyclecond}.
\end{lemma}

\begin{proof}
    Let $C=(e_0,\dots,e_{k-1})$ be a non-backtracking cycle. Summing
    \eqref{eq:potential} over the transitions $e_j\to e_{j+1}$ for
    $j=0,\dots,k-1$, with indices modulo $k$, the contributions of $\varphi$
    cancel, and we obtain $\sum_{j=0}^{k-1}\log\outdeg(e_j)=k\log\Lambda$, as claimed.
\end{proof}

\begin{lemma}\label{lem:spc-nec}
    If an NB-irreducible graph $G$ satisfies $\rho(G)=\Lambda(G)$, then it
    satisfies the suspended path condition \eqref{eq:suscond}.
\end{lemma}

\begin{proof}
    Given an NB-irreducible graph $G$ such that $\rho(G)=\Lambda(G)$, let $\varphi$ be the function given by Lemma~\ref{lem:potential}.
    For a vertex $v \in V(G)$, denote its incoming and outgoing edges by $\dirE_{\mathrm{in}}(v) = h^{-1}(v)$ and $\dirE_{\mathrm{out}}(v) = t^{-1}(v)$.
    
    Then, for any pair of edges $e \in \dirE_{\mathrm{in}}(v)$ and $f \in \dirE_{\mathrm{out}}(v)$ with $f\neq e'$, we have $e \rightarrow f$ and therefore by Lemma~\ref{lem:potential}
    \[
        \varphi(f)=\varphi(e)+\log\Lambda(G)-\log(\deg(v)-1).
    \]
    If $\deg(v) \ge 3$, then for any $e_1,e_2 \in \dirE_{\mathrm{in}}(v)$, there is an edge $f \in \dirE_{\mathrm{out}}(v)$ so that both $e_1 \rightarrow f$ and $e_2 \rightarrow f$, implying that $\varphi(e_1)=\varphi(e_2)$.
    Therefore, $\varphi$ is constant on $\dirE_{\mathrm{in}}(v)$, and by a similar argument, it is constant on $\dirE_{\mathrm{out}}(v)$ as well. 
    We denote these values by $\varphi_{\mathrm{in}}(v)$ and $\varphi_{\mathrm{out}}(v)$ respectively and obtain the following identity, for any $v \in V(G)$ with $\deg(v) \ge 3$:
    \begin{equation}\label{eq:in_out_potential}
            \varphi_{\mathrm{out}}(v)=\varphi_{\mathrm{in}}(v)+\log\Lambda(G)-\log(\deg(v)-1).
    \end{equation}

    Let $P = (e_0,\dots,e_{\ell-1}) \in \mathcal{S}$, let $u = t(e_0)$ and $w = h(e_{\ell-1})$, where $u,w$ are not necessarily distinct vertices. 
    By definition, $\deg(u),\deg(w)\ge 3$ and $\outdeg(e_i)=1$ for $i<\ell-1$.
    Applying equation \eqref{eq:potential} repeatedly along $P$ and along its inverse $P^{-1}$ gives:
    \begin{align*}
        \varphi_{\mathrm{in}}(w) &= \varphi_{\mathrm{out}}(u) + (\ell-1) \log\Lambda(G),\\
        \varphi_{\mathrm{in}}(u) &= \varphi_{\mathrm{out}}(w) + (\ell-1) \log\Lambda(G).
    \end{align*}
    Summing the two equations yields
    \[
        0 = \left[\varphi_{\mathrm{out}}(u) - \varphi_{\mathrm{in}}(u)\right] + \left[\varphi_{\mathrm{out}}(w) - \varphi_{\mathrm{in}}(w)\right] + 2(\ell-1) \log\Lambda(G).
    \]
    Applying \eqref{eq:in_out_potential} results in
    \[
        0 = 2\log\Lambda(G) - \log(\deg(u)-1)  - \log(\deg(w)-1) + 2(\ell-1) \log\Lambda(G),
    \]
    i.e. $(\deg(u)-1)(\deg(w)-1) = \Lambda(G)^{2\ell}$, which is the suspended path condition \eqref{eq:suscond} since $\indeg(P)=\deg(u)-1$ and $\outdeg(P)=\deg(w)-1$.
\end{proof}

\begin{proof}[Proof of Theorem~\ref{theorem:main}]
    $2 \implies 1$ and $3 \implies 1$ are Lemmas~\ref{lemma:suspended_implies_rho_eq_Lambda} and~\ref{lem:cycle-suff};
    $1 \implies 2$ and $1 \implies 3$ are Lemmas~\ref{lem:spc-nec} and~\ref{lem:cycle-nec}.
\end{proof}
\subsection{A relaxed form of the suspended path condition}\label{subsect:relaxed}

For an NB-irreducible graph $G$ define $g:\mathcal{S} \rightarrow\R$ by
\begin{equation}\label{eq:g_definition}
    g(P) = \bigl[\outdeg(P)\cdot\indeg(P)\bigr]^{\frac{1}{2|P|}}
\end{equation}
Then, given an NB-irreducible graph $G$, the suspended path condition \eqref{eq:suscond} is equivalent to the condition 
that $g(P) = \Lambda(G)$ for all suspended paths $P \in \mathcal{S}$.
The following corollary shows that it suffices to verify that $g$ is
constant, without computing $\Lambda(G)$.

\begin{corollary}\label{cor:realxedSuspended}
    For an NB-irreducible graph $G$,
    \begin{equation}\label{eq:lambda-g}
        \Lambda(G) = \prod_{P\in \mathcal{S}}g(P)^{|P|/|\dirE|}.
    \end{equation}
    Consequently, $\rho(G)=\Lambda(G)$ if and only if $g$ is constant on $\mathcal{S}$, in
    which case the common value is $\Lambda(G)$.
\end{corollary}

\begin{proof}
    Since $\mathcal{S}$ partitions $\dirE$, we have $\sum_{P\in \mathcal{S}}|P|=|\dirE|$, and since all
    edges of a suspended path except the last one have out-degree one,
    $\prod_{e\in\dirE}\outdeg(e)=\prod_{P\in \mathcal{S}}\outdeg(P)$. As $P\mapsto P^{-1}$ is
    a bijection of $\mathcal{S}$ with $\outdeg(P^{-1})=\indeg(P)$, we also have
    $\prod_{P\in \mathcal{S}}\outdeg(P)=\prod_{P\in \mathcal{S}}\indeg(P)$. Therefore
    \begin{eqnarray*}
        \Lambda(G) 
            &=& \prod_{e \in \dirE} {\outdeg(e)}^\frac{1}{|\dirE|}
             =  \prod_{P \in \mathcal{S}} {\outdeg(P)}^\frac{1}{|\dirE|}\\
            &=& \prod_{P \in \mathcal{S}} \Bigl[ \outdeg(P)\indeg(P) \Bigr] ^ \frac{1}{2|\dirE|}
             =  \prod_{P \in \mathcal{S}} {g(P)} ^ \frac{|P|}{|\dirE|}. 
    \end{eqnarray*}
    If $\rho=\Lambda$ then $g\equiv\Lambda$ by Theorem~\ref{theorem:main}. Conversely,
    if $g\equiv\alpha$ on $\mathcal{S}$, then as \eqref{eq:lambda-g} is a weighted geometric average of $g(P)$, we must have $\alpha = \Lambda$, so the suspended path condition
    holds, and $\rho=\Lambda$ by Theorem~\ref{theorem:main}.
\end{proof}

\section{Graphs satisfying $\rho = \Lambda$}\label{sect:graphs_with_rho_equal_lambda}

In this Section we consider some graphs and graph families with $\rho = \Lambda$.
\begin{claim}\label{claim:regular_biregular}
    If $G$ is an NB-irreducible graph with $\rho(G) = \Lambda(G)$, and 
    $G$ has no degree two vertices, 
    then $G$ is either a regular graph, or 
    a bipartite bi-regular graph.
\end{claim}

\begin{proof}
    If $G$ has no degree two vertices then all suspended paths have length one.
    Therefore, for any edge $e \in \dirE$ the product $\outdeg(e) \cdot \indeg(e)$ is fixed. 
    This implies that for any vertex $v$, all its neighbors must have the same degree. 
    As $G$ is connected, $G$ must be either regular or bipartite bi-regular.
\end{proof}

\begin{claim}\label{claim:subdivision}
    For any fixed $m$, if $G$ is an NB-irreducible graph with 
    $\rho(G)=\Lambda(G)$, and $G^m$ is obtained by replacing each edge
    in $G$ by a length $m$ path, then $\rho(G^m)=\Lambda(G^m)=\sqrt[m]{\rho(G)}$.
\end{claim}

\begin{proof}
    It is easy to verify that $G^m$ is NB-irreducible.
    As any suspended path $P \in \mathcal{S}(G)$ corresponds to a suspended path $P^m \in \mathcal{S}(G^m)$, 
    where $\indeg(P^m) = \indeg(P)$, $\outdeg(P^m) = \outdeg(P)$ and $|P^m| = m|P|$, 
    it follows that $g(P^m)=(g(P))^{1/m}$, where $g$ is defined by \eqref{eq:g_definition}.
    Therefore, as $g(P) = \Lambda(G)$ for all $P$, we have $g(P^m) = \Lambda(G)^{1/m}$. The claim follows using Corollary~\ref{cor:realxedSuspended}.
\end{proof}

We now give some other examples for graphs satisfying the suspended path condition, such as the three graphs in Figure~\ref{fig:lambda_rho_counterexample}.
    Consider the wheel with spokes graph $W_n$. 
    This graph is obtained by adding a central vertex $v_0$ to the cycle $C_n$ and connecting $v_0$ to all the cycle vertices $v_1,\ldots,v_n$.
    The next step is to obtain the graph $W_{n,\ell_1,\ell_2}$ by replacing each of the cycle edges of $W_n$ by a length $\ell_1$ path, and each spoke edge, i.e. an edge between the central vertex $v_0$ to a cycle vertex, by a length $\ell_2$ path. 
    Then the graph $W_{n,\ell_1,\ell_2}$ has the following types of suspended paths:
    \begin{itemize}
        \item Cycle suspended paths $P_c$ of length $\ell_1$, with $\indeg(P_c)=\outdeg(P_c)=2$, implying that $g(P_c) = (2 \cdot 2)^{1/(2\ell_1)}$.
        \item Spoke suspended paths $P_s$ of length $\ell_2$, oriented from $v_0$ with $\indeg(P_s)=n-1$ and $\outdeg(P_s)=2$, implying that $g(P_s) = g(P_s') = (2 \cdot (n-1))^{1/(2\ell_2)}$.
    \end{itemize}
    Therefore, by Corollary~\ref{cor:realxedSuspended}, $\rho(W_{n,\ell_1,\ell_2}) = \Lambda(W_{n,\ell_1,\ell_2})$ iff
    \begin{equation}\label{eq:spokes_graph_condition}
        (2 \cdot 2)^{1/(2\ell_1)} = (2 \cdot (n-1))^{1/(2\ell_2)}.
    \end{equation}
    Setting $n=2^k+1$, this equality is equivalent to $\ell_1(1+k) = 2 \ell_2$.
    We further require that $\ell_1$ and $\ell_2$ are relatively prime, otherwise the graph can be obtained from  a smaller graph by Claim~\ref{claim:subdivision}.
    It follows that for even $k$, one must have $\ell_1 = 2$, $\ell_2 = k+1$ and for odd $k$ one must have $\ell_1 = 1$ and $\ell_2 = (k+1)/2$.
    This yields an infinite family of graphs, which we denote by $H_k$, all satisfying \eqref{eq:spokes_graph_condition} and hence $\rho(H_k)=\Lambda(H_k)$.
    For $k>1$, none of the $H_k$ graphs falls into the categories stated by Claims~\ref{claim:regular_biregular}, \ref{claim:subdivision}.
    The three graphs (a), (b) and (c) in Figure~\ref{fig:lambda_rho_counterexample} are respectively 
    $H_2 = W_{5,2,3}$, $H_3 = W_{9,1,2}$, $H_4 = W_{17,2,5}$.             

\section{Variance of the Number of Random Bits}\label{sect:variance}

\subsection{Variance Dichotomy for General Graphs}\label{subsect:dichotomy}

Consider the random variable $R_\ell$, 
counting the number of random bits used by a length $\ell$ NBRW on $G$, starting from the stationary distribution $\nu_s$. 
Formally, for $\omega = (e_0,e_1,\ldots,e_\ell) \in \Omega_\ell$ we define $R_\ell(\omega) = \sum_{i=0}^{\ell-1} \log_2 \outdeg(e_i)$.
The following theorem states that the variance of $R_\ell$ exhibits a dichotomy.

\begin{theorem}\label{theorem:variance_dichotomy}
    Given an NB-irreducible graph $G$, then:
    \begin{equation*}
        \var[R_\ell] = \begin{cases}
            O(1)     & \mbox{ if }\rho = \Lambda \\
            \Theta(\ell) & \mbox{ if }\rho > \Lambda.
        \end{cases}
    \end{equation*}
\end{theorem}

We split the proof into several lemmas. In Lemma~\ref{lemma:variance_if_rho_eq_Lambda} we handle the case $\rho = \Lambda$. 
In Lemma~\ref{lemma:two_cycles_variance} we prove that  if $\rho > \Lambda$ then there are two equal length cycles that have different geometric average out-degree.
In Lemma~\ref{lemma:variance_lb} we prove the variance lower bound when $\rho > \Lambda$,
and in Lemma~\ref{lemma:variance_ub} we prove the variance upper bound when $\rho > \Lambda$.

\begin{lemma}\label{lemma:variance_if_rho_eq_Lambda}
    If $G$ is an NB-irreducible graph with $\rho = \Lambda$ then there is a constant $c=c(G)$ so that $\var[R_\ell] \leq c$ for all $\ell$.
\end{lemma}

\begin{proof}
    Since $\rho=\Lambda$, Theorem~\ref{theorem:main} implies that the suspended path condition holds.
    Therefore, given an NB-path $\omega = (e_0,\ldots,e_\ell)\in \Omega_\ell$, by Lemmas~\ref{lemma:suspended_implies_rho_eq_Lambda} and~\ref{lemma:bounded},
    $\left\vert R_\ell(\omega) - \ell \cdot \log_2 \Lambda \right\vert = O(1)$.
    As $\E[R_\ell] = \ell \cdot \log_2 \Lambda$, this implies that $\var[R_\ell] = O(1)$, as claimed.    
\end{proof}

\begin{lemma}\label{lemma:two_cycles_variance}
    Given an NB-irreducible graph $G$ with $\rho > \Lambda$, there exist two non-backtracking cycles $C_1$ and $C_2$ with $|C_1| = |C_2|$, 
    a common edge in $C_1 \cap C_2$, but different average out-degree:
    \begin{equation*}
        \left[ \prod_{e \in C_1} \outdeg(e) \right]^{\dfrac{1}{|C_1|}} < \left[ \prod_{e \in C_2} \outdeg(e) \right]^{\dfrac{1}{|C_2|}}.
    \end{equation*}
\end{lemma}

\begin{proof}
    Since $\rho>\Lambda$, Theorem~\ref{theorem:main} implies that the cycle condition does not hold.
    Therefore, there is a cycle $C$ with $\prod_{e \in C} \outdeg(e) \neq \Lambda(G)^{|C|}$.
    We only prove the claim for the case
    \begin{equation}
        \frac{1}{|C|} \sum_{e \in C} \log_2 \outdeg(e) < \log_2 \Lambda(G),    
    \end{equation}
    as proof for the other case is virtually identical.
    Let $f$ be some edge in $C$.
    Then, the first step is to find a second cycle $C'$ with a strictly larger average log out-degree such that $f \in C'$.
    Indeed, since $\E[R_\ell] = \ell \log_2 \Lambda$, then for any $\ell$, there is some specific walk $\omega=(e_0,e_1,\ldots,e_\ell) \in \Omega_\ell$ with $R_\ell(\omega) \geq \ell \log_2 \Lambda$.
    Therefore, since $G$ is NB-irreducible, there are paths $P_1,P_2$ respectively from $f$ to $e_0$ and from $e_\ell$ to $f$, 
    whose length is bounded by some constant $l_0$ that depends only on $G$. 
    Therefore, $C' = P_1 \omega P_2$ is a cycle with $f \in C'$.
    Furthermore, 
    \begin{equation*}
               \frac{1}{|C'|} \sum_{e \in C'} \log_2 \outdeg(e)
        \geq    \frac{1}{\ell + 2l_0} R_\ell(\omega)
        \geq   \frac{\ell}{\ell + 2l_0} \log_2\Lambda.
    \end{equation*}
    By taking a sufficiently large $\ell$, one can ensure that:
    \begin{equation}
        \frac{1}{|C|} \sum_{e \in C} \log_2 \outdeg(e) < \frac{1}{|C'|} \sum_{e \in C'} \log_2 \outdeg(e).    
    \end{equation}
    We end the construction by looping through $C$ and $C'$ the required amount of times to obtain the cycles $C_1$ and $C_2$, whose length is  the LCM of $|C|$ and $|C'|$,
    meeting the requirements of the lemma.
\end{proof}

\begin{lemma}\label{lemma:variance_lb}
    If $G$ is an NB-irreducible graph with $\rho > \Lambda$ then there is a constant $c=c(G)$ so that $\var[R_\ell] \geq c \cdot \ell$ for all $\ell$.
\end{lemma}

\begin{remark}
    In the proof of Lemma~\ref{lemma:variance_ub} we derive an exact formula~\eqref{eq:exact_normalized_variance} for $\lim_{\ell\to\infty}\var[R_\ell]/\ell$.
    As it is unclear how to directly prove this limit is non-zero if $\rho>\Lambda$, we resort to cruder techniques in the following proof.
\end{remark}

\begin{proof}[Proof of Lemma~\ref{lemma:variance_lb}]
    Given an NB-irreducible graph $G$ with $\rho > \Lambda$,
    by Lemma~\ref{lemma:two_cycles_variance} 
    there are two cycles $C_1$ and $C_2$ of the same length $\ell_0=|C_1| = |C_2|$, 
    with different average out-degree, and with some edge $f$ in their intersection.
    Recall that:
    \begin{equation*}
        \var[R_\ell] = \E_{\omega \in \Omega_\ell}\left[ {(R_\ell(\omega) - \ell \log_2\Lambda)}^2\right].
    \end{equation*}
    In order to lower-bound the variance, we use the exposure method. 
    We expose some of the edges of $\omega = (e_0,e_1,\ldots,e_\ell)$
    and argue that the conditional expectation of ${(R_\ell(\omega) - \ell \log_2\Lambda)}^2$ given the exposed edges is $\Omega(\ell)$ with high probability.
    Define the mapping $\phi:\dirE\to\N$, where $\phi(e)$ is the length of the shortest NB-path from $e$ to $f$ if $e \neq f$ and $\phi(f) = \ell_0$.
    Also, let $\ell_{\max}$ be the maximum of $\phi(e)$ on $e \in \dirE$.
    We build the partial exposure mapping $\chi$ from $\{0,\ldots,\ell\}$ to $\dirE$ by the following randomized procedure:
    
    \begin{center}
        \fbox{\begin{minipage}{7 cm}
        \begin{algorithmic}
            \State $\chi \gets $ the empty mapping
            \State $i \gets 0$
            \While{$i < \ell$}
                \State Expose $e_i$ and set $\chi(i) \gets e_i$
                \State $i \gets i + \phi(e_i)$
            \EndWhile
            \State Expose $e_\ell$ and set $\chi(\ell) \gets e_\ell$
        \end{algorithmic}
        \end{minipage}}        
    \end{center}
    
    The domain of $\chi$ is the set of exposed edge indices, $i_1=0 < i_2 < \cdots < i_k=\ell$.
    Therefore, conditioned on $\chi$, one may write $R_\ell$ as the sum of independent random variables:
    \begin{equation*}
        R_\ell = \sum_{i=0}^{\ell-1} \log_2 \outdeg(e_i) = \sum_{j=1}^{k-1} R(i_{j+1}-i_j,e_{i_j},e_{i_{j+1}}),
    \end{equation*}
    where $R(\ell',e',e'')$ is the random variable counting the number of random bits used by a length $\ell'$ NBRW, conditioned on it beginning at $e'$ and ending at $e''$.
    Therefore:
    \begin{equation*}
        \var[R_\ell | \chi]
           = \sum_{j=1}^{k-1} \var\!\left[R(i_{j+1}-i_j,e_{i_j},e_{i_{j+1}})\right]
           \geq |I| \cdot \var\!\left[R(\ell_0,f,f)\right],
    \end{equation*}
    where $I = I(\chi) = \{i \in  \dom(\chi) :\, i+\ell_0 \in \dom(\chi) \mbox{ and } \chi(i) = \chi(i+\ell_0) = f\}$.
    As the number of exposed edges is lower bounded by $\ell/\ell_{\max} = \Omega(\ell)$,
    it follows that $|I|$ is lower bounded by a binomial random variable $X \sim B(\ell/(2\ell_{\max}),p_{\min})$, 
    where $p_{\min}$ is a positive constant defined as the minimum of $(\Pi^{\phi(e)})_{e,f} \cdot (\Pi^{\ell_0})_{f,f}$ on $e \in \dirE$. 
    Therefore, $\E[|I|] \geq \E[X] = c_1 \ell$, where $c_1 = p_{\min}/(2 \ell_{\max})$.
    The Chernoff bound on the lower tail of $X$ yields $\Pr[|I| \leq c_1 \ell/2] \leq \Pr[X \leq c_1 \ell/2] \leq \exp[-c_1 \ell/2]$, which is $o(1)$ as $\ell$ grows.

    Recall that by our assumption on the cycles $C_1, C_2$, the variance of $R(\ell_0,f,f)$ is some positive constant $c_2$.
    Therefore:
    \begin{eqnarray*}
        \var[R_\ell]
        &=& \E\!\left[{(R_\ell(\omega) - \ell \log_2\Lambda)}^2\right]
        = \sum_\chi \E\!\left[{(R_\ell(\omega) - \ell \log_2\Lambda)}^2 \Big\vert \chi\right] \cdot \Pr[\chi] \\
        &\geq& \sum_\chi |I(\chi)| \cdot \var\!\left[R(\ell_0,f,f)\right] \cdot \Pr[\chi] \\
        &\geq& \frac{c_1 \ell}{2} \cdot \var\!\left[R(\ell_0,f,f)\right] \cdot \Pr\!\left[|I| \geq \frac{c_1 \ell}{2}\right] \\
        &\geq& \frac{c_1 c_2 \ell}{2} \cdot \Pr\!\left[|I| \geq \frac{c_1 \ell}{2}\right]
        = \Omega(\ell).
    \end{eqnarray*}
\end{proof}

\begin{lemma}\label{lemma:variance_ub}
    If $G$ is an NB-irreducible graph then there is a constant $c=c(G)$ so that $\var[R_\ell] \leq c \cdot \ell$ for all $\ell$.
\end{lemma}

\begin{proof}
    We define $R'_\ell = R_\ell - \ell \log_2\Lambda = \sum_{i=0}^{\ell-1} R'_{\ell,i}$, 
    where $R'_{\ell,i}$ are dependent, zero mean, identically distributed random variables with $R'_{\ell,i}(\omega) = \log_2 \outdeg(e_i) - \log_2\Lambda$.
    Therefore:
    \begin{eqnarray*}
        \var[R_\ell] 
        &=& \var[R'_\ell]
        =   \E\!\left[ {(R'_\ell)}^2\right] 
        =   \sum_{i,j=0}^{\ell-1} \E\!\left[ R'_{\ell,i} R'_{\ell,j}\right] \\
        &=& \ell \cdot \E\!\left[ {(R'_{\ell,0})}^2\right] + 2 \sum_{\Delta=1}^{\ell-1} (\ell-\Delta) \cdot \E\!\left[ R'_{\ell,0} R'_{\ell,\Delta}\right],
    \end{eqnarray*}
    where in the last step we use the fact that the NBRW starting from the stationary distribution 
    $\nu_s$ is time-invariant.

    Let $f$ denote a column vector where $f_e = \log_2 \outdeg(e) - \log_2\Lambda$ and note that $\sum_e f_e = 0$. 
    Recall that $\Pi=\Pi_G$ is the NBRW transition matrix.
    Then:
    \begin{equation}\label{eq:normalized_variance}
        \frac{\var[R_\ell]}{\ell} = \frac{1}{|\dirE|}  \cdot f^T f + \frac{2}{|\dirE|} \sum_{\Delta=1}^{\ell-1} (1-\frac{\Delta}{\ell})  \cdot f^T \Pi^\Delta f.            
    \end{equation}
    Let $J$ be the Jordan normal form of $\Pi$, written as the direct sum of its blocks $J=J_1 \oplus J_2 \ldots \oplus J_k$, 
    where block $J_i$ corresponds to the eigenvalue $\lambda_i$.
    Assume further that $\lambda_1 = 1 \geq |\lambda_2| \geq \cdots \geq |\lambda_k|$ and that $\Pi = Q^{-1} J Q$.
    Then:
    \begin{eqnarray*}
        \frac{\var[R_\ell]}{\ell} 
        &=& - \frac{1}{|\dirE|} \cdot f^T f + \frac{2}{|\dirE|} f^T Q^{-1} \left( \sum_{\Delta=0}^\ell (1-\frac{\Delta}{\ell}) \cdot  J^\Delta  \right) Q f\\
        &=& - \frac{1}{|\dirE|} \cdot f^T f + \frac{2}{|\dirE|} f^T Q^{-1} \left[\bigoplus_{i=1}^k \left( \sum_{\Delta=0}^\ell (1-\frac{\Delta}{\ell}) \cdot  J_i^\Delta  \right) \right] Q f  \\
        &=& - \frac{1}{|\dirE|} \cdot f^T f + \frac{2}{|\dirE|} f^T Q^{-1} \left[0 \oplus \bigoplus_{i=2}^k \left( \sum_{\Delta=0}^\ell (1-\frac{\Delta}{\ell}) \cdot  J_i^\Delta  \right) \right] Q f,  
    \end{eqnarray*}
    where in the last equality  we can replace $J_1$ by the scalar zero. This is justified because the leftmost column of $Q^{-1}$ is a multiple of the all ones vector, orthogonal to $f^T$.
    The next step is to employ the following identity, holding for any matrix $A$ provided that $I-A$ is invertible:
    \begin{equation}\label{eq:matrix_power_sum}
        \sum_{\Delta=0}^\ell (1-\frac{\Delta}{\ell}) \cdot A^\Delta = \frac{1}{\ell} \cdot (A^{\ell+1} - A)(I-A)^{-2} + (I-A)^{-1}.
    \end{equation}
    As $I-J_i$ is invertible for all $i>1$, it follows that the normalized variance can be written as:
    \begin{align}
        \frac{\var[R_\ell]}{\ell} = & -\frac{1}{|\dirE|} \cdot f^T f + \\
        &\frac{2}{|\dirE|} f^T Q^{-1} \left[0 \oplus \bigoplus_{i=2}^k \left( \frac{1}{\ell} \cdot (J_i^{\ell+1} - J_i)(I-J_i)^{-2} + (I-J_i)^{-1} \right) \right] Q f. \nonumber     
    \end{align}
    The matrices $J_i^{\ell+1}$ are upper triangular Toeplitz, 
    where $(J_i^{\ell+1})_{u,v} = \lambda_i^{\ell+1-v+u} \binom{\ell+1}{v-u}$ for $u \leq v$. 
    Therefore, if $|\lambda_i| < 1$, all entries of the matrix $J_i^{\ell+1}$ are bounded in absolute value by some constant that is a function of $G$ only.
    Otherwise, if $|\lambda_i|=1$, then necessarily $\lambda_i$ has algebraic multiplicity of one (see~\cite{horn2012matrix}, Corollary~8.4.6), so $J_i$ is a scalar and  $|J_i^{\ell+1}|=1$.
    It follows that the normalized variance converges when $\ell\to\infty$:
    \begin{equation}\label{eq:exact_normalized_variance}
        \lim_{\ell\to\infty} \frac{\var[R_\ell]}{\ell} 
        = - \frac{1}{|\dirE|} \cdot f^T f + \frac{2}{|\dirE|} f^T Q^{-1} \left[0 \oplus \bigoplus_{i=2}^k (I-J_i)^{-1} \right] Q f.
    \end{equation}
    As the right-hand side does not depend on $\ell$ this concludes the proof.
\end{proof}

\subsection{The Case of $K_4$ Minus an Edge}\label{subsect:K4_minus_edge}
    Consider the graph of $K_4$ minus an edge, depicted in Figure~\ref{fig:k4_minus_edge}. 
    The graph has 10 directed edges, where $\outdeg(e)$ is two on six of the edges and one on the remaining four edges,
    implying that the NBRW consumes 0.6 bits per step on average, and that $\Lambda=2^{6/10}\approx 1.5157$. 
    The value of $\rho$ is the Perron eigenvalue of the 10 by 10 non-backtracking adjacency matrix $B$ as stated in Lemma~\ref{lemma:rho_Perron}, 
    which can be reduced to the following 3 by 3 matrix as there are only three edge types, $(u_.,u_.),(v_.,u_.)$ and $(u_.,v_.)$:
    \begin{equation*}
        \rho(B) = \rho\left( \left[
            \begin{tabular}{ccc} 0&0&2\\ 1&0&1\\ 0&1&0 \end{tabular}
        \right] \right) \approx 1.5214
    \end{equation*}
    Therefore, $\rho>\Lambda$ for this graph, implying by Theorem~\ref{theorem:main}, that the suspended path condition as well as the cycle condition are violated.
    Indeed, the suspended path condition is violated as the length one suspended path $P=[(u_1,u_2)]$ has 
    \begin{equation*}
        \left(\outdeg(P) \cdot \indeg(P)\right)^{\frac{1}{2|P|}} = (2\cdot2)^\frac{1}{2} = 2 \neq \Lambda.
    \end{equation*}
    The length three cycle $C=[(u_1,u_2),(u_2,v_1),(v_1,u_1)]$ violates the cycle condition as:
    \begin{equation*}
        \left( \prod_{e \in C} \outdeg(e) \right)^{\frac{1}{|C|}} = (2 \cdot 1 \cdot 2)^{\frac{1}{3}} = 2^{\frac{2}{3}}\neq \Lambda=2^{\frac{3}{5}}.
    \end{equation*}
    
    \begin{figure}[h]
        \centering
        \includegraphics[width=0.25\textwidth]{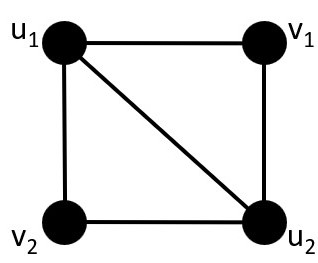}
        \caption{$K_4$ minus an edge.}
        \label{fig:k4_minus_edge}
    \end{figure}
    
    Next, we regard the random variable $R_\ell$ counting the number of random bits used by a length $\ell$ NBRW starting from the stationary distribution on the $K_4$ minus an edge graph, as defined in the beginning of Subsection~\ref{subsect:dichotomy}.
    By Corollary~\ref{cor:Rl}, the expected number of random bits per step is $\E[R_\ell/\ell] = \log_2\Lambda = 0.6$,
    while $\log_2 \rho = \lim_{\ell \rightarrow \infty} \frac 1 \ell \log_2 \E[2^{R_\ell}] \approx 0.605$. 

    This can be demonstrated by an experiment.
    In Figure~\ref{fig:bit_histogram} we depict the PDF of $R_\ell/\ell$ for $K_4$ minus an edge for $\ell=1000$.
    As expected by the Markov chains Central Limit Theorem, for example~\cite{aldous-fill-2014} chapter 2,
    the PDF of $R_\ell/\ell$ can be approximated by a normal distribution with mean $\log_2\Lambda = 0.6$.
    The observed normalized variance $\var[R_\ell]/\ell$ is very close to $2/125=0.016$, in agreement with the exact limit computed in Lemma~\ref{lemma:k4_minus_edge_variance} below.
    By Theorem~\ref{theorem:variance_dichotomy}, $\var[R_\ell]=\Theta(\ell)$, implying that as $\ell$ grows, $\var[R_\ell/\ell]$ goes down as $1/\ell$. Therefore, the limit of $\frac {1}{\ell}  \log_2 \E\!\left[2^{R_\ell}\right]$ is governed by the extreme upper tail of the PDF of $R_\ell/\ell$.
    \begin{figure}[h]
        \centering
        \includegraphics[width=1.0\textwidth,clip,trim=1mm 1mm 1mm 1mm]{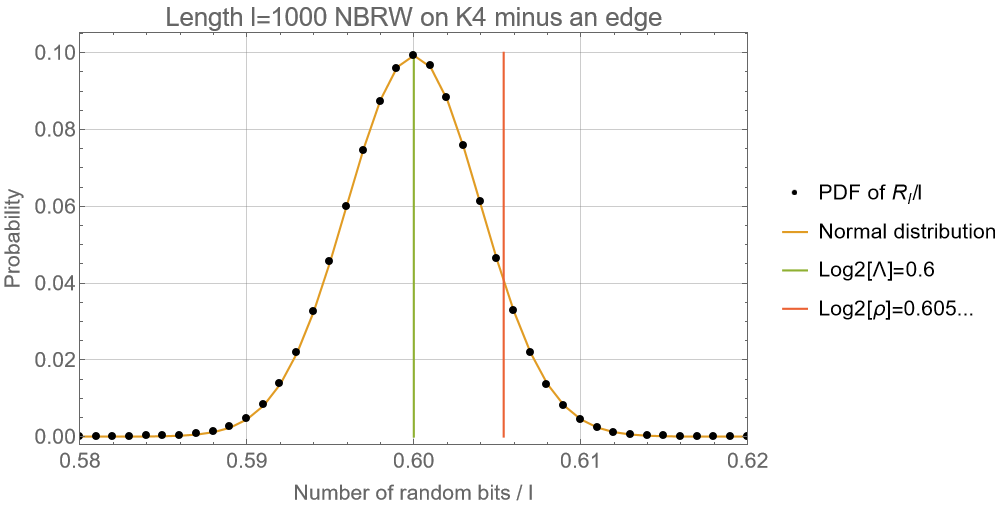}
        \caption{
            The number of bits per step, $R_\ell/\ell$, for the length $\ell=1000$ NBRW on $K_4$ minus an edge. 
            The black dots are the exact PDF, the yellow line is a best-fit normal distribution, 
            the green line is at $\log_2\Lambda$ 
            and the orange line is at $\log_2\rho$ which is the limit on $\ell$ of $\frac {1}{\ell} \log_2 \E\!\left[2^{R_\ell}\right]$.
        }
        \label{fig:bit_histogram}
    \end{figure}

    We end with an exact computation of the asymptotic normalized variance of the number of bits consumed by the NBRW on $K_4$ minus an edge, which is a specialization of the computation in Lemma~\ref{lemma:variance_ub} to the specific graph.
    \begin{lemma}\label{lemma:k4_minus_edge_variance}
        Given the probability space $\Omega_\ell$ of the length $\ell$ NBRW on $K_4$ minus an edge, starting from the stationary distribution, the random variable  
        $R_\ell(\omega) = \sum_{i=0}^{\ell-1} \log_2 \outdeg(e_i)$, for $\omega=(e_0,\ldots,e_\ell) \in \Omega_\ell$ satisfies:
        \begin{equation}
            \lim_{\ell\to\infty} \frac{\var[R_\ell]}{\ell} = \frac{2}{125}
        \end{equation}
    \end{lemma}
    
\begin{proof}
    As in Lemma~\ref{lemma:variance_ub}, we define $R'_\ell = \sum_{i=0}^{\ell-1} R'_{\ell,i}$, 
    where $R'_{\ell,i}$ are random variables with $R'_{\ell,i}(\omega) = \log_2 \outdeg(e_i) - \log_2\Lambda$.
    Therefore:
    \begin{equation*}
        \frac{\var[R_\ell]}{\ell} = - \E\!\left[ {(R'_{\ell,0})}^2\right] + 2 \sum_{\Delta=0}^{\ell} (1-\frac{\Delta}{\ell}) \cdot \E\!\left[ R'_{\ell,0} R'_{\ell,\Delta}\right].
    \end{equation*}

    In order to avoid 10x10 matrices, we exploit the fact that $K_4$ minus an edge has only three edge types.
    Then, the stationary distribution, the NBRW transition matrix, and the vector $f$ with $f_e = \log_2 \outdeg(e) - \log_2\Lambda$, 
    with row and column order of: $uu$, $vu$ and $uv$, are as follows:
    \begin{equation*}
        \nu_s = \begin{pmatrix} 1/5 \\ 2/5 \\ 2/5 \end{pmatrix}, \,
        \Pi = \begin{pmatrix} 0 & 0 & 1 \\ 1/2 & 0 & 1/2 \\ 0 & 1 & 0 \end{pmatrix}, \,
        f = \begin{pmatrix} 2/5 \\ 2/5 \\ -3/5\end{pmatrix}.
    \end{equation*}
    Using $N_s$ to denote the matrix with diagonal $\nu_s$ we write:
    \begin{equation}
        \frac{\var[R_\ell]}{\ell} = - f^T N_s f + 2 \sum_{\Delta=0}^{\ell} (1-\frac{\Delta}{\ell})  \cdot f^T N_s \Pi^\Delta f.
    \end{equation}
    We next observe that the matrix $\Pi$ is diagonalizable, $\Pi = Q^{-1} D Q$, with:
    \begin{equation*}
        Q = \begin{pmatrix} \frac{1}{2} & 1 & 1 \\ i & -1-i & 1 \\ -i & -1+i & 1 \end{pmatrix}, \,
        D = \begin{pmatrix} 1 & 0 & 0 \\ 0 & -\frac{1}{2}+\frac{i}{2} & 0 \\ 0 & 0 & -\frac{1}{2}-\frac{i}{2} \end{pmatrix},
    \end{equation*}
    where the first row of $Q$ is the left Perron eigenvector, equal to $\nu_s$ up to scaling.
    In addition, the right Perron eigenvector, which is the first column of $Q^{-1}$, is necessarily the all one vector up to scaling.
    Therefore, $(f^T N_s Q^{-1})_1 = 0$, implying that $f^T N_s Q^{-1} D = f^T N_s Q^{-1} E$, where $E$ is the matrix obtained by zeroing the first diagonal entry of $D$.
    Plugging this into the variance expression yields:
    \begin{eqnarray*}
        \frac{\var[R_\ell]}{\ell} 
        &=& - f^T N_s f + 2 \sum_{\Delta=0}^{\ell} (1-\frac{\Delta}{\ell}) \cdot f^T N_s Q^{-1} D^\Delta Q f \\
        &=& - f^T N_s f + 2 f^T N_s Q^{-1} \left( \sum_{\Delta=0}^\ell (1-\frac{\Delta}{\ell}) \cdot E^\Delta \right) Q f.
    \end{eqnarray*}
    Using~\eqref{eq:matrix_power_sum}, the normalized variance expression converges to:
    \begin{eqnarray*}
        \lim_{\ell\to\infty} \frac{\var[R_\ell]}{\ell} 
        &=& -f^T N_s f + 2 f^T N_s Q^{-1} \left( \lim_{\ell\to\infty} \sum_{\Delta=0}^\ell \frac{\ell-\Delta}{\ell} \cdot E^\Delta \right) Q f \\
        &=& -f^T N_s f + 2 f^T N_s Q^{-1} (I-E)^{-1} Q f.
    \end{eqnarray*}
    Plugging in the numbers, yields $2/125$ on the right hand side, as claimed.
\end{proof}
    
\section{Observations and Open Questions}
\begin{enumerate}
    \item
        Is there a quantitative version of our main result of Theorem~\ref{theorem:main}? 
        Something along the lines of: {\em a graph is $\epsilon$-close to satisfying the suspended path condition if and only if $\rho \leq \Lambda + \delta(\epsilon)$}. It seems relatively easy to prove that being close to satisfying the suspended path condition implies that $\rho$ is close to $\Lambda$. However, the other direction is not obvious, if true at all.
    \item
        What is the minimal possible value $\rho$ may have, given the degree distribution?
        (The maximal value of $\rho$ seems to be as close as we wish to the maximal degree minus one, in the degree distribution support.)
    \item 
        Specifically, what is the minimal possible value of $\rho$, for a graph with half the vertices being of degree two and the other half of degree three. 
        We conjecture that the minimum is $\rho(K_4 \mbox{ minus an edge})$.
    \item 
        Prove that any girth $g$ graph, with half its vertices being of degree two and the other half of degree three, has 
        $g \leq (2+o(1)) \log_\rho n$, where $\rho = \rho(K_4 \mbox{ minus an edge})$.
    \item 
        Can one directly see why violation of the suspended path condition necessarily implies that the limit of $\var[R_\ell]/\ell$ is strictly positive in~\eqref{eq:exact_normalized_variance}?
    \item 
        What is the minimal possible value of $\lim_{\ell \to \infty}\var[R_\ell]/\ell$ for a graph $G$ with a given degree distribution?
        Is the minimum obtained on the same graphs that minimize $\rho$?
    \item
        Is there an  NB-irreducible graph $G$ for which the NBRW transition matrix $\Pi_G$ is not diagonalizable?
\end{enumerate}

\section{Acknowledgement}
    We would like to thank Ofer Zeitouni for numerous fruitful discussions. 
    The proof that $\rho=\Lambda$ implies the cycle condition 
    was inspired by his arguments and by the book of Dembo and Zeitouni on large deviations~\cite{dembo2009large}.
\bibliographystyle{abbrv}
\bibliography{EntrpGrtRtRef}

@misc{aldous-fill-2014,
 AUTHOR = {Aldous, David and Fill, James Allen},
 TITLE = {Reversible {M}arkov Chains and Random Walks on Graphs},
  YEAR = {2002},
  NOTE = {Unfinished monograph, recompiled 2014, available at \url{http://www.stat.berkeley.edu/~aldous/RWG/book.html}}
}

@article{alon2002moore,
  title={The {M}oore bound for irregular graphs},
  author={Alon, Noga and Hoory, Shlomo and Linial, Nathan},
  journal={Graphs and Combinatorics},
  volume={18},
  number={1},
  pages={53--57},
  year={2002},
  publisher={Springer}
}

@article{alon2007non,
  title={Non-backtracking random walks mix faster},
  author={Alon, Noga and Benjamini, Itai and Lubetzky, Eyal and Sodin, Sasha},
  journal={Communications in Contemporary Mathematics},
  volume={9},
  number={04},
  pages={585--603},
  year={2007},
  publisher={World Scientific}
}

@article{angel2015non,
  title={The non-backtracking spectrum of the universal cover of a graph},
  author={Angel, Omer and Friedman, Joel and Hoory, Shlomo},
  journal={Transactions of the American Mathematical Society},
  volume={367},
  number={6},
  pages={4287--4318},
  year={2015}
}

@article{ben2017cutoff,
  title={Cutoff for nonbacktracking random walks on sparse random graphs},
  author={Ben-Hamou, Anna and Salez, Justin},
  journal={The Annals of Probability},
  volume={45},
  number={3},
  pages={1752--1770},
  year={2017},
  publisher={Institute of Mathematical Statistics}
}

@phdthesis{conchon2021sparse,
  title={Sparse random graphs: from local specifications to global phenomena},
  author={Conchon-Kerjan, Guillaume},
  year={2021},
  school={Universit{\'e} de Paris/Universit{\'e} Paris Diderot (Paris 7)}
}

@article{conchon2022cutoff,
  title={Cutoff for random lifts of weighted graphs},
  author={Conchon-Kerjan, Guillaume},
  journal={Annals of Probability},
  volume={50},
  number={1},
  pages={304--338},
  year={2022}
}

@book{dembo2009large,
  title={Large deviations techniques and applications},
  author={Dembo, Amir and Zeitouni, Ofer},
  volume={38},
  year={2009},
  publisher={Springer Science \& Business Media}
}

@book{glover2021non,
  title={The Non-Backtracking Spectrum of a Graph and Non-Backtracking PageRank},
  author={Glover, Cory},
  year={2021},
  publisher={Brigham Young University}
}

@article{hoory2002size,
  title={The size of bipartite graphs with a given girth},
  author={Hoory, Shlomo},
  journal={Journal of Combinatorial Theory, Series B},
  volume={86},
  number={2},
  pages={215--220},
  year={2002},
  publisher={Elsevier}
}

@article{hoory2024girth,
  title={On the Girth of Graph Lifts},
  author={Hoory, Shlomo},
  journal={arXiv preprint arXiv:2401.01238},
  year={2024}
}

@article{kingman1961convexity,
  title={A convexity property of positive matrices},
  author={Kingman, John FC},
  journal={The Quarterly Journal of Mathematics},
  volume={12},
  number={1},
  pages={283--284},
  year={1961},
  publisher={Oxford University Press}
}

@article{lubetzky2010cutoff,
  title={Cutoff phenomena for random walks on random regular graphs},
  author={Lubetzky, Eyal and Sly, Allan},
  journal={Duke Mathematical Journal},
  volume={153},
  number={3},
  pages={475--510},
  year={2010},
  publisher={Duke University Press}
}

@article{lubetzky2016cutoff,
  title={Cutoff on all {R}amanujan graphs},
  author={Lubetzky, Eyal and Peres, Yuval},
  journal={Geometric and Functional Analysis},
  volume={26},
  number={4},
  pages={1190--1216},
  year={2016},
  publisher={Springer}
}

@article{nussbaum1986convexity,
  title={Convexity and log convexity for the spectral radius},
  author={Nussbaum, Roger D},
  journal={Linear Algebra and its Applications},
  volume={73},
  pages={59--122},
  year={1986},
  publisher={Elsevier}
}

@book{horn2012matrix,
  title={Matrix analysis},
  author={Horn, Roger A and Johnson, Charles R},
  year={2012},
  edition={2nd},
  publisher={Cambridge university press}
}

\end{document}